\documentclass{amsart}
\usepackage{color}
\usepackage{graphicx}
\usepackage{amsmath,amscd}
\usepackage{amssymb}
\input epsf
\usepackage{epsfig}
\usepackage{stmaryrd}

\usepackage{enumerate}
\newtheorem{theorem}{Theorem}[section]
\newtheorem{lemma}[theorem]{Lemma}
\newtheorem{corollary}[theorem]{Corollary}

\newtheorem{proposition}[theorem]{Proposition}
\theoremstyle{definition}
\newtheorem{definition}[theorem]{Definition}

\theoremstyle{remark}

\numberwithin{equation}{section}

\newcommand{\Real}{{\mathbb R}}

\newcommand{\eps}{\varepsilon}

\newcommand {\hide}[1]{}

\begin{document}
\title[Lojasiewicz inequality] {An exponential {\L}ojasiewicz inequality for semi-Pfaffian sets}
\author{Nicolai Vorobjov}
\address{Nicolai Vorobjov\\
St.Petersburg Department of Steklov Mathematical Institute} 
\email{nicolaivorobjov@gmail.com}

\begin{abstract}
We prove a generalization of the version, due to A. Gabrielov \cite{G}, of the {\L}ojasiewicz inequality for
not necessarily restricted semi-Pfaffian sets.
Unlike the most variants of the {\L}ojasiewicz inequality, it describes the rate of growth of a Pfaffian function
on a semi-Pfaffian set $X$ not only in a neighbourhood of its zero set in the domain of $X$ but also in a
neighbourhood of zeroes on the boundary of the domain.
Gabrielov's version is essentially one-dimensional, we generalize it to a multidimensional case.
\end{abstract}

\maketitle

\section{Introduction}

The {\L}ojasiewicz inequality is a fundamental tool in real analytic geometry.
It establishes a lower bound on the rate of growth of an analytic function in the neighbourhood of its zeroes.
Many versions of the inequality are known, applicable to settings varying from semialgebraic category to
o-minimal structures.
In \cite{G}, for the needs of the theory of Pfaffian {\em limit sets}, the following version of
the {\L}ojasiewicz inequality was proved
(we will remind necessary definitions from the theory of Pfaffian functions in the next section.)

In what follows, $\overline X$ denotes the topological closure of a set $X \subset \Real^n$ in the ambient space.
For $r \in {\mathbb N}$ define the iterated exponential function as $e_r(t)= \exp(e_{r-1}(t))$ with $e_0(t)=t$.

\begin{proposition}[\cite{G}, Proposition~2.12]\label{prop:one-dim}
Let $G \subset \Real^n$ be the domain of a Pfaffian chain of order $r$.
Let $X \subset G$ be a (not necessarily restricted) semi-Pfaffian set with respect to this chain,
and let $g:\ G \to \Real$ be Pfaffian with respect to the same chain. 
Suppose that ${\bf 0} \in \overline{X \cap \{ g>0 \}}$.
Then
\begin{equation}\label{eq:gab2.12}
{\bf 0} \in \overline{\left\{  x \in X,\  g(x) \ge \frac{1}{e_r(\| x \|^{-N})} \right\}}
\end{equation}
for some $N \in {\mathbb N}$.
\end{proposition}

The method of the proof of this proposition in \cite{G} is as follows.
First, using a version of the Curve Selection lemma (see \cite[Lemma~2.11]{G}), the statement is reduced to the case
when $X \cap \{ g>0 \}$ is a smooth curve.
More precisely, it is sufficient to prove that, for a branch $\gamma$ of this curve with
${\bf 0} \in \overline \gamma$, there is a neighbourhood $U$ of ${\bf 0}$ such that for every $x \in \gamma \cap U$
we have
\begin{equation}\label{eq:gab2.12v}
g(x) \ge \frac{1}{e_r(\| x \|^{-N})}
\end{equation}
for some $N \in {\mathbb N}$.
The proof of (\ref{eq:gab2.12v}) is then deduced from the growth properties of functions in Hardy fields
due to M. Rosenlicht \cite{R}.

Note, that in the case of Pfaffian functions defined in $\Real$, inequalities similar to (\ref{eq:gab2.12v})
were established by D. Grigoriev \cite{Gr}.
An important version of the {\L}ojasiewicz inequality in Pfaffian setting, with the bound involving iterated exponential
functions, is proved in \cite{L}.

In this paper we generalize Proposition~\ref{prop:one-dim} to the case when the point ${\bf 0}$ in the formulation
is replaced by $A \subset \Real^n$, in first instance a semi-Pfaffian closed smooth manifold,
and then a semi-Pfaffian compact smooth manifold with boundary.
In these generalizations, the semi-Pfaffian set $X$ is not necessarily restricted and $A$ may lie in the boundary
of the domain in which $X$ is defined.
Possible further generalizations of the type of $A$ would probably require resolution of some difficult questions,
concerning tubular neighbourhoods of stratified sets, which is likely to overshadow the actual
Lojasiewicz-type inequality results.

\section{Preliminaries}\label{sec:prelim}

We recall some basic definitions and properties of Pfaffian functions and semi- and sub-Pfaffian sets.
A more comprehensive exposition can be found in \cite{GV, Kh}.

\begin{definition}[\cite{GV}]
A {\em Pfaffian chain} of the order $r \ge 0$ and degree $\alpha \ge 1$ in an open domain $G \subset \Real^n$ is 
a tuple of analytic functions $f_1, \ldots , f_r$ in $G$ satisfying differential equations
$$
\frac{\partial f_j}{\partial x_i}= P_{i,j}(x, f_1(x), \ldots, f_j(x))
$$
for $1 \le j \le r,\ 1 \le i \le n$.
Here $P_{i,j}(x, y_1, \ldots, y_j)$ are polynomials in variables $x=(x_1, \ldots ,x_n),\ y_1, \ldots ,y_j$
of degrees not exceeding $\alpha$.
A function
$$f(x)=P(x, f_1(x), \ldots, f_r(x)),$$
where $P(x, y_1, \ldots ,y_r)$ is a polynomial of a degree not exceeding $\beta \ge 1$, is called
a {\em Pfaffian function} of order $r$ and degree $(\alpha, \beta)$. 
\end{definition}

Examples of Pfaffian functions include multivariate polynomials, logarithm in $G=\{ x>0 \}$, trigonometric functions
in appropriate bounded domains.
The function $e_r(t)$ is Pfaffian in $G= \Real$ of order $r$ and degree $(r,1)$.

\begin{definition}[\cite{GV}]
A set $X \subset \Real^n$ is called {\em semi-Pfaffian} in an open domain $G \subset \Real^n$ if it consists of
points in $G$ satisfying a Boolean combination of some equations and inequalities of the kind $f=0$, $g>0$,
where $f,g$ are Pfaffian functions having a common Pfaffian chain defined in $G$.
A semi-Pfaffian set $X$ is {\em restricted in} $G$ if $\overline X \subset G$.
A semi-Pfaffian set is called {\em basic} if the Boolean
combination is just a conjunction of equations and strict inequalities.
\end{definition}

\begin{definition}[\cite{GV}]
A set $X \subset \Real^n$ is called {\em sub-Pfaffian} if it is an image of a semi-Pfaffian set $Y \subset \Real^{n+m}$
under a projection into a subspace.
Sub-Pfaffian $X$ is called {\em restricted} if there exists a pre-image $Y$ which is a restricted semi-Pfaffian.
\end{definition}

A restricted sub-Pfaffian set need not be semi-Pfaffian, by the famous example of Osgood (\cite{O}, see also \cite{GV}).
In particular, quantifier elimination is not generally possible in the first order theory of the reals
with added Pfaffian functions.
Restricted sub-Pfaffian sets form an o-minimal structure \cite{Gcompl, GV}, for sub-Pfaffian sets that
are not restricted, this is currently not known.

Pfaffian functions, semi- and sub-Pfaffian sets and maps have global finiteness properties similar to
the properties of polynomials and semialgebraic sets.
In what follows, we will use, without special provisions, these finiteness properties, in particular will adjust
some general theorems (e.g., Sard's and Fubini's theorems) to the o-minimal setting.
 
In many instances it is possible to find explicit upper bounds on various characteristics of semi- and sub-Pfaffian
sets as functions of orders of Pfaffian chains and degrees of defining Pfaffian functions.
In this paper we will be interested only in the dependence on the order, which will always be a constant function.

We now recall two properties of semi-Pfaffian sets, needed in the proof of the main theorem.

\begin{proposition}[\cite{G2}]\label{prop:closure}
Let $X \subset G \subset \Real^n$ be a semi-Pfaffian set in a domain $G$ having the Pfaffian chain of order $r$.
Then both $\overline X \cap G$ and $\partial X= \overline{X \cap G} \setminus X$ are semi-Pfaffian sets in $G$,
having Pfaffian chains of order at most $r$.
\end{proposition}

\begin{definition}\label{def:nonsing}
A basic semi-Pfaffian set $X \subset \Real^n$, with ${\rm codim}\ X=k$, is {\em effectively nonsingular}, if
the system of equations and inequalities defining $X$ includes a set
of $k$ Pfaffian functions $h_{1}, \ldots , h_{k}$ such that the restriction $h_{j}|_{X} \equiv 0$ for
$1 \le j \le k$ and $dh_{1} \land \cdots \land dh_{k} \neq 0$ at every point of $X$.
\end{definition}

\begin{definition}[\cite{GV-strat}]\label{def:strat}
A {\em weak stratification} of a semi-Pfaffian set $X$ is a partition of $X$ into a
disjoint union of smooth (i.e., nonsingular), not necessarily connected, semi-Pfaffian subsets $X_i$ called {\em strata}.
A stratification is {\em basic} if all strata are basic semi-Pfaffian sets which are {\em effectively nonsingular}.
Observe, that if $X$ is compact, then the boundary of each stratum is contained in (but not necessarily coincides with)
the union of some other strata.
In what follows, we will call a basic weak stratification of $X$ just a {\em stratification} of $X$.
\end{definition}

\begin{proposition}[\cite{GV-strat}]\label{prop:strat}
Let $X \subset \Real^n$ be a semi-Pfaffian set having the Pfaffian chain of order $r$.
There exists a stratification of $X$ (in the sense of Definition~\ref{def:strat}) such that each stratum
is a semi-Pfaffian set having Pfaffian chain of order at most $r$.
\end{proposition}

\section{Main result}

Suppose that $g:\ \Real^n \to \Real$ is a Pfaffian function of order $s$.
Let $A \subset \{ g=0 \}$ be a closed (compact, without boundary) smooth manifold which is an
{\em effectively non-singular} basic
semi-Pfaffian set (see Definition~\ref{def:nonsing}) with the same, as $g$, Pfaffian chain of order $s$.
Let $G \subset \Real^n$ be the domain of a Pfaffian chain of order $r$.
Let $X \subset G$ be a (not necessarily restricted) semi-Pfaffian set with respect to this chain.
For $a \in A$, let $X_\eps(a)$ be set of all vectors $x-a$ in the intersection of the normal space
to $A$ at $a$ with $X \cap \{ g>0 \}$, having norm $\| x-a \|=\eps$.

For each $a \in A$ and each $\eps >0$ the set $X_\eps(a)$ is semi-Pfaffian (with the order of
Pfaffian chain at most $r+s$), because the normal space is the span of gradients of functions defining
$A$ and can be expressed in terms of ranks of matrices.

\begin{theorem}\label{th:main2}
Let $\overline X \subset ( (\{g>0 \} \cap G)\cup A)$. 
Then there is a neighbourhood $U$ of $A$ in $X \cap \{ g>0 \}$ and there is $N \in {\mathbb N}$ such that
for every $x \in U$  the inequality
\begin{equation}\label{eq:bound2}
g(x) \ge \frac{1}{e_{r+s}({\rm dist} (x, A)^{-N})}
\end{equation}
holds true.
\end{theorem}

\begin{proof}
First notice, that if $A \cap \overline{X \cap \{ g>0 \}} = \emptyset$, then, by compactness of $A$,
there is a neighbourhood of $A$ containing no points of $X \cap \{ g>0 \}$, and the theorem is vacuous. 

Introduce new variables $y=(y_1, \ldots, y_n)$.
Let $\widehat X_\eps \subset \Real^{2n}$, where $\Real^{2n}$ is the space of coordinates
$x_1, \ldots ,x_n,y_1, \ldots, y_n$, be the set coinciding with $X_\eps(a)$ for every fixed $y=a \in A$.
Observe, that for all sufficiently small $\eps >0$,
$$
\widehat X_\eps= \{ (x,y) \in X \times A|\> y\ \text{is nearest to}\ x\ \text{among points in}\ A,\ \| x-y \|= \eps \},
$$
noting that the relation ``$y$ is nearest to $x$'' is equivalent to $x-y \in N_y A$, where
$N_y A$ is the normal space to $A$ at $y$.

The set $\widehat X_\eps \subset \Real^{2n}$ is semi-Pfaffian with the combined Pfaffian chain of $g$ and $X$,
having the order at most $r+s$.
Denote by $\pi:\ \Real^{2n} \to \Real^n$ the projection map on the subspace of variables $x_1, \ldots, x_n$.
Observe, that for every point $x$ in the neighbourhood of $A$ in $X \cap \{ g>0 \}$ there is $a \in A$ (the closest
point to $x$ in $A$) and there is $\eps >0$ (the distance from $x$ to $A$) such that $x \in X_\eps(a)$.
It follows, that for all sufficiently small $\eps >0$, the sub-Pfaffian set
$X_\eps := \pi (\widehat X_\eps)$ is the set of all points in $X \cap \{ g>0 \}$
at the distance $\eps$ from $A$, i.e., distance-$\eps$ level set from $A$ in $X \cap \{ g>0 \}$.

Introduce the function $\widehat g:\ \Real^{2n} \to \Real$ defined as $(x,y) \mapsto g(x)$ for every $y \in \Real^n$. 
Since $\overline{\widehat X_\eps}$ is compact for each sufficiently small $\eps>0$, function $\widehat g$
attains minimum on $\overline{\widehat X_\eps}$.
Denote the set of all these minima by $\widehat M_\eps$.
It follows that for each sufficiently small $\eps>0$, function $g$ attains minimum on $\overline{X_\eps}$, and
the set of all these minima is $M_\eps:= \pi (\widehat M_\eps)$.
Observe, that sets $\widehat M_\eps$ and $M_\eps$ are restricted sub-Pfaffian.

Introduce the set
$$
Z:= \left\{(x,y)|\> x \in \overline{X \cap {\{ g>0 \}}} \setminus A,\ y \in A,\ x-y \in N_y A \right\}
$$
Due to Proposition~\ref{prop:closure}, $Z$ is a semi-Pfaffian set, while,
due to Proposition~\ref{prop:strat}, there is an effectively non-singular stratification ${\mathcal S}_1$ of $Z$.
Also, by Proposition~\ref{prop:strat}, all strata are semi-Pfaffian sets with the combined Pfaffian chain of $g$ and $X$,
having the order at most $r+s$.
Define $\widehat A:= \{ (x,y) \in A^2|\> x=y \} \subset \Real^{2n}$ and $\rho(x,y):= \|x-y \|^2$.
Choose initially a neighbourhood $U$ of $\widehat A$ such that for each stratum $S$,
$$
S \cap U \neq \emptyset\quad \text{implies}\quad \widehat A \cap \overline{S \cap U} \neq \emptyset.
$$
(Note, that here and further $U$ can be chosen to be semi-Pfaffian, with the Pfaffian chain of order at most $r+s$,
by adding the inequality $\rho (x,y) < \delta$ for sufficiently small $\delta >0$.)

Consider $\widehat Y:=S \cap U$, where $S$ is one of the strata in ${\mathcal S}_1$ with $S \cap U \neq \emptyset$,
such that $\widehat Y_\eps:= \widehat Y \cap \{ \rho (x,y)= \eps^2 \}$
contains minimum of $\widehat g$ on $\overline{\widehat X_\eps}$ for all sufficiently small $\eps> 0$.

Such $\widehat Y$ exists.
Indeed, let
$$
\widehat M:= \bigcup_{0 < \eps \le \eps_0} \widehat M_\eps,
$$
where $\eps_0$ is small.
For every $\delta >0$, for every $S$ in ${\mathcal S}_1$ the truncated set
$$\widehat M \cap S \cap \{ \rho(x,y) > \delta^2 \}$$
is restricted sub-Pfaffian.
Therefore, there exists an upper bound on the number of connected components of this set,
not depending on $\delta$ (see \cite{GV-cd}).
It follows that for every $S$ either there exists $\delta>0$ such that  all connected components of $\widehat M \cap S$
stay a positive distance away from $\widehat A$, or there is a connected component $M'$ of $\widehat M \cap S$ such that
$\overline{M'} \cap \widehat A \neq \emptyset$.
The first alternative can't be true for all $S$ since that would contradict the fact that for each sufficiently
small $\eps>0$, function $\widehat g$ attains minimum on $\overline{\widehat X_\eps}$.
Hence, there is a stratum $S$ for which the second alternative is valid, and we take $\widehat Y= S \cap U$.
Then $\widehat Y_\eps= \widehat Y \cap \{ \rho (x,y)= \eps^2 \}$ is as required, i.e.,
contains minimum of $\widehat g$ on $\overline{\widehat X_\eps}$ for all sufficiently small $\eps> 0$.
Indeed, since $M'$ is connected and $\rho$ in continuous on $M'$, $\rho (M')$ is an interval.
Moreover, $\rho$ is positive on $Z$, while $0 \in \overline{\rho (M')}$ because
$\overline{M'} \cap \widehat A \neq \emptyset$.
Hence, $(0, \eps') \subset \rho (M')$ for some $\eps'>0$.

Since $\widehat Y$ is a smooth manifold by construction, and $\rho|_{\widehat Y}$ is smooth
(as a restriction of smooth function to smooth submanifold), the set
$\widehat Y_\eps:= \widehat Y \cap \{ \rho= \eps^2 \}$ is a smooth manifold by Sard's theorem.

We can also assume that sets $\widehat M \cap \widehat Y_\eps$ are compact
for all sufficiently small $\eps >0 $.
Otherwise, the frontiers $\overline{\widehat M \cap \widehat Y_\eps} \setminus (\widehat M \cap \widehat Y _\eps)$
for some $\eps >0$ would lie in the union $W$ of some adjacent to $\widehat Y$ strata of smaller dimension.
If each such frontier lies outside a neighbourhood of $\widehat A$, then we decrease $U$ to fit into this neighbourhood.
Else, we replace $\widehat Y$ by one of the strata in $W$, and continue the argument inductively.
It follows that the sets of minima of $\widehat g$ contained in $\widehat Y_\eps$ can be assumed to be compact.

The restriction $\widehat g|_{\widehat Y_\eps}$ is smooth for sufficiently small $\eps >0$.
Minima of $\widehat g$ on $\widehat Y_\eps$ are attained at critical points of $\widehat g|_{\widehat Y_\eps}$.

Define
\begin{equation}\label{eq:critical}
\widehat C:= \left\{ (x,y) \in \widehat Y|\ \nabla \widehat g(x,y) \in \left( N_{(x,y)} \widehat Y +
{\rm span} \{ \nabla \rho (x,y) \}\right)\right\},
\end{equation}
where $+$ stands for the sum of two subspaces.
Observe, that the condition on $\nabla \widehat g(x,y)$ in (\ref{eq:critical}) is the critical point
condition for $\widehat g|_{\widehat Y_\eps}$ at $(x,y)$ for every sufficiently small $\eps >0$, and it is
independent of $\eps$.
It follows, that $\widehat C_\eps:= \widehat C \cap \{ \rho= \eps^2 \}$ is the set of all critical
points of $\widehat g|_{\widehat Y_\eps}$.
In particular, $\widehat C_\eps$ contains all minima of $\widehat g$ on $\widehat Y_\eps$
for all sufficiently small $\eps >0$. 
The set $\widehat C$ can be defined in terms of ranks of some matrices, and therefore is semi-Pfaffian,
with Pfaffian chain of the order at most $r+s$.

Consider the effectively non-singular stratification ${\mathcal S}_2$ of $\widehat C$.
All strata of ${\mathcal S}_2$ are semi-Pfaffian sets,
with the Pfaffian chain of the order at most $r+s$.
By Sard's theorem applied to function $\rho(x,y)$, the stratification ${\mathcal S}_2$
induces the effectively non-singular stratification
${\mathcal S}_3$ on $\widehat C_\eps$ for all sufficiently small $\eps >0$.
Since each stratum $S$ of ${\mathcal S}_3$ is of the form $S=S' \cap \{ \rho= \eps^2 \}$, where
$S'$ is a stratum of ${\mathcal S}_2$, all strata of ${\mathcal S}_3$ also are semi-Pfaffian sets,
having the Pfaffian chain of order at most $r+s$.

Generally, on every connected component of the set of critical points of a smooth function on a smooth manifold
the function has constant value.
It follows, that for every connected component $\widehat D_\eps$ of $\widehat C_\eps$
the restriction $\widehat g|_{\widehat D_\eps}$ is constant.
Let $\widehat D'_\eps$ be a component on which the global minimum of $\widehat g|_{\widehat C_\eps}$ is attained.
The component $\widehat D'_\eps$ is compact (since we assumed that the sets of minima of $\widehat g$,
contained in $\widehat Y_\eps$, are compact) and is a union of some connected components of some strata of
stratification $\widehat {\mathcal S}_3$.

Introduce the function $\psi:= c(x,y)$ on $\Real^{2n}$, where $c=(c_1, \ldots, c_{2n}) \in \Real^{2n}$
is a generic vector.
The restriction $\psi|_{\widehat D'_\eps}$ has a global extremum on compact set $\widehat D'_\eps$,
hence a critical point on at least one of connected components of a stratum of $\widehat{\mathcal S}_3$
contained in $\widehat D'_\eps$.
We claim that this critical point is non-degenerate, thus isolated (the following proof of this claim
is an adjustment of the proof of Lemma~2.15 in \cite{G}).

Indeed, let $S$ be a stratum in ${\mathcal S}_2$ and $S_\eps= S \cap \{ \rho = \eps^2 \}$ for sufficiently small $\eps >0$
be the corresponding stratum in ${\mathcal S}_3$.
Define
$$
F_{S, \eps}:= \{ c \in \Real^{2n}|\ \psi = c(x,y) \text{ has a degenerate critical point on}\ S_\eps \}.
$$
By Lemma~2.14 in \cite{G}, this set has measure zero in $\Real^{2n}$.

Let
$$
F_S:= \bigcup_{\eps \in (0, \eps_0)} (F_{S, \eps}, \eps)= \{ (c, \eps)|\ c \in F_{S, \eps} \}
\subset \Real^{2n} \times (0, \eps_0).
$$
Due to Fubini theorem, $F_S$ has measure zero in $\Real^{2n+1}$.
It follows, that for a generic $c$, the set
$$
E_{S,c}:= \{ \eps|\ (c, \eps) \in F_S \}= F_S \cap \{ c= {\rm const}\}
$$
has measure zero in $\Real$.
Since there are finitely many strata $S$ in ${\mathcal S}_2$, one generic $c$ guaranties
the set $E_{S,c}$ to have measure zero for every $S$.
It follows that
$$
E_c= \bigcup_{S \in {\mathcal S}_2} E_{S,c}
$$
has measure zero in $\Real$, hence is a finite set.
Thus, after sufficiently decreasing $\eps_0$, we get $E_c \cap (0, \eps_0)= \emptyset$.
It follows, that for every sufficiently small $\eps >0$ all critical points of $\psi$ on every stratum in
${\mathcal S}_3$ are non-degenerate.
In particular, critical points contained in $\widehat D'_\eps$ are nondegenerate.
The claim is proved.

Critical points of $\psi$ contained in sets of the kind
$\widehat D'_\eps$, belong to finite sets, satisfying the condition
$$
\nabla \psi \in \left( N_{(x,y)} S' + {\rm span} \{ \nabla \rho \} \right),
$$
where $S'$ is a stratum in ${\mathcal S}_2$.
This condition is independent of $\eps$.
Define
$$
\Gamma:= \bigcup_{S \in {\mathcal S}_2}\left\{ (x,y) \in S|\ \nabla \psi (x,y) \in \left( N_{(x,y)} S +
{\rm span} \{ \nabla \rho \} \right) \right\}.
$$
This is a semi-Pfaffian subset of $S$ having the Pfaffian chain of order at most $r+s$.
We claim that $\Gamma$ is one-dimensional set (a curve) such that
for some neighbourhood $\widehat U$ of $\widehat A$ the intersection $\Gamma \cap \widehat U$
is a disjoint union of smooth connected curves for each of which its closure intersects $\widehat A$.
Indeed, for every sufficiently small $\eps>0$ the set $\Gamma \cap \{ \rho= \eps^2 \}$ is finite and non-empty.
For all sufficiently small $0< \delta_0 < \eps_0$ the set $\Gamma \cap \{ \delta_0^2 < \rho < \eps_0^2 \}$
is restricted sub-Pfaffian.
It follows, by Hardt's triviality theorem \cite{C, vdD}, that this set is one-dimensional for all
sufficiently small $0< \delta_0 < \eps_0$, and, therefore, four a neighbourhood $\widehat U$,
the set $\Gamma \cap \widehat U$ is as required.

Now we argue as above (in the proof of existence of suitable $\widehat Y$).
Consider truncations $\Gamma \cap \widehat U \cap \{ \rho > \delta \}$ for sufficiently small $\delta >0$,
and note that the number of connected components of critical points of $\psi$, belonging to $\Gamma \cap \widehat U$,
is bounded from above independently of $\delta$.
It follows, that for some neighbourhood $\widehat U$ of $\widehat A$ in $Z$, 
at least one smooth curve in $\Gamma \cap \widehat U$, say $\gamma$, consists of points on which a global minimum
of $\widehat g|_{\overline{\widehat X_\eps} \cap \widehat U }$ for all sufficiently small $\eps >0$ is attained.
Finiteness properties of semi-Pfaffian sets imply, that $\overline \gamma \cap \widehat A$ is a single point,
let it be $(a,a)$.

Applying the proof of Proposition~1.2 to the smooth semi-Pfaffian curve $\gamma$, after translation by $(a,a)$,
we obtain $N' \in {\mathbb N}$ such that for all points of $\gamma$ sufficiently close to $(a,a)$,
the inequality
\begin{equation}\label{eq:bound3}
\widehat g(x,y) \ge \frac{1}{e_{r+s}(\|(x,y)-(a,a)\|^{-N'})}
\end{equation}
holds true. 

Since, by Cauchy--Schwarz inequality,
$${\rm dist} (x, A)=\| x-y \| \le \sqrt{2}\ \|(x,y)-(a,a)\|$$
and the function
$$
\varphi(z)=\frac{1}{e_{r+s}(z^{-N'})}
$$
monotonically increases with respect to $z$, the bound
(\ref{eq:bound3}) implies
$$
g(x) \ge \frac{1}{e_{r+s}\left( \left( \frac{1}{\sqrt{2}}\ {\rm dist} (x, A)\right )^{-N'} \right)}
$$
for every $x \in \pi(\gamma)$.
By choosing an integer $N$ sufficiently larger than $N'$, we get the inequality (\ref{eq:bound2}) for
every $x \in \pi(\gamma)$.
It follows that for all sufficiently small $\eps >0$, a point $x \in \pi(\gamma \cap \overline{\widehat X_\eps})$
satisfies the inequality (\ref{eq:bound2}).
On the other hand, for all $x' \in X_\eps$, we have $g(x) \le g(x')$.
It follows, that every $x' \in X_\eps$ satisfies (\ref{eq:bound2}) for the same $N \in {\mathbb N}$
and the neighbourhood $U= \{ x \in X \cap \{ g>0 \} |\ {\rm dist} (x,A) < \eps_0 \}$
for a sufficiently small $\eps_0 >0$.
\end{proof}

\section{Corollary}

In this section we generalize Theorem~\ref{th:main2} for the case when $A$ is a smooth manifold with boundary.

\begin{corollary}\label{cor:cor}
Theorem~\ref{th:main2} remains valid when $A$ is a compact smooth manifold with boundary
such that both $A \setminus \partial A$ and $\partial A$ are effectively nonsingular basic semi-Pfaffian sets
having the common Pfaffian chain of order $s$.
More precisely, assume that $A= \{ F_1= \cdots =F_q=0, H \le 0 \}$, $\partial A= A \cap \{ H=0\}$, where
$F_i, H$ are Pfaffian functions of order $s$,
and $dF_1 \land \cdots \land dF_q \land dH \neq 0$ at every point of $A$.
\end{corollary}

Let $A \subset \ \Real^n$ be a compact smooth effectively non-singular semi-Pfaffian manifold with boundary
$B:=\partial A$.
Let $\mathring A:= A \setminus B$.
For $a\in A$, define the {\em outward normal cone} $C_aA$ as follows.
If $a\in \mathring A$, put $C_aA:=N_aA$.
If $a\in B$, let $\eta_a\in T_aA$ be an inward-pointing vector orthogonal to $T_aB$, and define
$C_aA:=N_aA+ \Real_{\ge0}(-\eta_a)$.
 
\begin{lemma}\label{le:lemma}
\begin{enumerate}
\item
If $x\in \Real^n$, $a\in A$, and $\| x-a \|=\operatorname{dist}(x,A)$, then
$x-a\in C_aA$.

\item
The set $CA:=\{(a,v)\in A\times \Real^n|\> v\in C_aA \}$
is semi-Pfaffian, with respect to the same Pfaffian chain as $A$.
\end{enumerate}
\end{lemma}
 
 \begin{proof}
(1)\quad
The claim is obvious when $a\in\mathring A$, and we used it already at the beginning of the proof of the theorem.

Fix $x \in \Real^n$.
Let $v:=x-a$ and define the function $\varphi_x(z):= \frac{1}{2} \| x-z \|^2$, $z \in \Real^n$.
The point $a \in A$ is a minimum of the restriction $\varphi_x|_A$.
Computing its derivative at $a$ in the direction of a vector $w \in \Real^n$, we get
$d \varphi_x(a)(w)=- \langle v,w \rangle$,
where $\langle \cdot \rangle$ is the dot product.

Suppose now that $a\in B$.
The tangent cone of $A$ at $a$ is $T_aB+\Real_{\ge0}\eta_a$.
For every vector $w$ in this tangent cone the one-sided derivative of $\varphi_x$ in the direction $w$ is non-negative.
Hence $-\langle v,w\rangle\ge0$.
Taking $w\in T_aB$, as well as $-w\in T_aB$, gives $v\perp T_aB$, i.e., $v \in N_a B$.
Taking inward direction $w=\eta_a$ gives $\langle v,\eta_a\rangle\le0$, since $\eta_a \in T_a A$.

Finally, $N_aB=N_aA\oplus\mathbb R\eta_a$.
Thus, $v=\alpha+ \lambda \eta_a$ for some $\alpha \in N_a A$ and $\lambda \in \Real$.
Since $\alpha \perp \eta_a$, we have $\langle v, \eta_a \rangle = \lambda \| \eta_a \|^2 \le 0$, hence $\lambda \le 0$.
It follows that $v=\alpha+ (-\lambda)(- \eta_a)$ with $-\lambda \ge 0$, and therefore
$v\in N_aA+\Real_{\ge 0}(-\eta_a)=C_aA$. 

(2)\quad Now we prove that $CA$ is semi-Pfaffian.

Since
$$
A= \{ F_1= \cdots =F_q=0, H \le 0 \},
$$
where $F_i, H$ are Pfaffian functions in $\Real^n$ such that
the rank of the Jacobian matrix, ${\rm rank} (\nabla F_1, \ldots , \nabla F_q)=q$ on $A$, while
$$
B= \partial A= \{ F_1= \cdots =F_q=H=0 \}
$$
with ${\rm rank} (\nabla F_1, \ldots, \nabla F_q, \nabla H)=q+1$ on $B$.
Then,
$$
\mathring A= \{ F_1= \cdots =F_q=0, H < 0 \}.
$$

For $a\in\mathring A$, the condition $v\in C_aA$ is
$$
v\in \operatorname{span} \{\nabla F_1(a),\ldots,\nabla F_q(a) \},
$$
and therefore is expressed by the vanishing of the appropriate
minors of the matrix $\bigl(\nabla F_1,\ldots,\nabla F_q,v \bigr)$.

Now let $a \in B$.

Since $H <0$ in $\mathring A$, the orthogonal projection $\eta_a$ of the vector $\nabla H (a)$ onto
$T_a A$  points outward, being non-zero and orthogonal to $T_a B$.
By the definition of $C_a A$, we have $C_a A= N_a A+ \Real_{\ge 0}(-\eta_a)$.
The issue is that $\eta_a$ is not immediately given by a semi-Pfaffian formula.

Since $N_a B= N_a A \oplus \Real \eta_a$, the equivalent description of $C_a A$ is
\begin{multline*}
C_a A= \{ v \in N_a B|\> \langle v, \eta_a \rangle \le 0 \}\\
= \left\{ v|\> v\in \operatorname{span} \{\nabla F_1(a),\ldots,\nabla F_q(a),\nabla H(a)\},\>
\langle v,\eta_a \rangle \le0 \right\}.
\end{multline*}
The linear-span condition in this formula is again given by rank conditions.
To see that the inequality $\langle v,\eta_a \rangle \le 0 $ is semi-Pfaffian, consider the Gram matrix
$$
{\mathcal G}(a)=(\langle\nabla F_i(a),\nabla F_j(a) \rangle)_{i,j=1}^q.
$$
Since vectors $\nabla F_i (a)$ are linearly independent on $A$, we have $D(a):= \det {\mathcal G}(a) >0$.

By the property of Gram matrix,
\begin{equation}\label{eq:gram}
\eta_a= -{\rm proj}_{T_a A} \nabla H(a)= -\nabla H(a)+ \sum_{i=1}^{q} \beta_i(a) \nabla F_i(a),
\end{equation}
for some coefficients $\beta (a)=(\beta_1(a), \ldots ,\beta_q(a))$ satisfying the system of linear equations
${\mathcal G}(a) \beta(a)^T= b(a)$, where
$$
b_i(a)=\langle \nabla H(a), \nabla F_i (a) \rangle\quad \text{and}\quad b(a)= (b_1(a), \ldots ,b_q(a))^T.
$$
Expressing these coefficients by the Cramer's rule, we get $\beta_i(a)= \Delta_i(a)/D(a)$,
where $\Delta_i(a)$ is the determinant obtained be replacing the $i$-th column in ${\mathcal G}(a)$ by $b(a)$.
Thus, from (\ref{eq:gram}) we get
$$
D(a) \eta_a= -D(a) \nabla H(a) + \sum_{i=1}^{q} \Delta_i(a) \nabla F_i(a).
$$
Observe, that the right-hand side of this expression is semi-Pfaffian with respect to the Pfaffian chain of $A$.

Since $D(a)>0$, the inequality $\langle v,\eta_a \rangle \le 0 $ is equivalent to $\langle v,D(a)\eta_a \rangle\le 0$.
Thus, $C_a A$, for every $a \in A$, is given by a semi-Pfaffian formula.
It follows that $C A$ is a semi-Pfaffian set.
\end{proof}

\begin{proof}[Proof of Corollary~\ref{cor:cor}]
The proof largely follows the proof of Theorem~\ref{th:main2}.

By Lemma~\ref{le:lemma}, the set $CA$ is semi-Pfaffian with respect to the same Pfaffian chain as $A$.
Introduce new variables $y=(y_1,\ldots,y_n)$ and define
$$
W:= \left\{ (x,y)\in \Real^{2n}: x\in \overline{X\cap{g>0}}\setminus A,\quad y\in A,\quad x-y\in C_yA \right\}.
$$
The set $W$ is semi-Pfaffian with respect to the combined Pfaffian chain of $X$ and $A$, hence of order at most $r+s$.
As in the proof of Theorem~\ref{th:main2}, put $\widehat g(x,y):=g(x),\ \rho(x,y):=\|x-y \|^2$.
For $\varepsilon>0$, let $W_\varepsilon:=W \cap \{\rho=\varepsilon^2 \}$.
For all sufficiently small $\varepsilon>0$, the set $W_\varepsilon$ is compact.
Indeed, the set $A$ is compact and $\|x-y \|=\varepsilon$, hence both $x$ and $y$ stay in a bounded neighbourhood of $A$.
Moreover, by the assumption $\overline X\subset ({g>0}\cap G)\cup A$, and because smooth manifold with boundary $A$
has positive reach (the nearest point on $A$ to $x$ is unique),
every limit point of the $x$-coordinates with positive distance from $A$ belongs to $\overline{X\cap{g>0}}\setminus A$.
Therefore, $\widehat g$ attains a minimum on $W_\varepsilon$.
Denote this minimum by $m(\varepsilon)$.
We now repeat the proof of Theorem~\ref{th:main2}, replacing the set $Z$ used there by $W$.
Take an effectively non-singular semi-Pfaffian stratification of $W$.
Using the same truncated connected-component argument as in the proof of Theorem~\ref{th:main2},
choose a stratum $\widehat Y$ such that $\widehat Y_\varepsilon:=\widehat Y\cap \{ \rho=\varepsilon^2 \}$
contains a point at which the global minimum $m(\varepsilon)$ is attained for every sufficiently small $\varepsilon>0$.

Define
$$
\widehat C:= \left\{ (x,y)\in\widehat Y|\> \nabla\widehat g(x,y) \in \left( N_{(x,y)}\widehat Y+
\operatorname{span}\{\nabla\rho(x,y)\} \right) \right\}.
$$
As in the proof of Theorem~\ref{th:main2}, $\widehat C$ is semi-Pfaffian, since the criticality condition
for $\widehat g|_{\widehat Y_\varepsilon}$ is expressed by rank conditions.
Stratify $\widehat C$ according to Proposition~\ref{prop:strat}.
Let $\psi(x,y)=c\cdot (x,y)$ for a generic $c\in\mathbb R^{2n}$.
By the same Sard--Fubini argument, as in the proof of Theorem~\ref{th:main2}, taking
$\varepsilon_0$ sufficiently small, all critical points of the restrictions of $\psi$ to the strata of
$\widehat C\cap \{\rho=\varepsilon^2 \}$ are isolated for every $0<\varepsilon<\varepsilon_0$.
Hence their total locus $\Gamma$ is a semi-Pfaffian set having finite non-empty fibers under $\rho$.
By Hardt triviality, every truncated set
$\Gamma\cap \{\delta^2<\rho<\varepsilon_0^2 \}$, where $\delta>0$, is one-dimensional.
Arguing as in the proof of Theorem~\ref{th:main2}, there is a smooth semi-Pfaffian curve
$\gamma\subset\Gamma$ such that for every sufficiently small $\varepsilon>0$ the set
$\gamma\cap \{\rho=\varepsilon^2 \}$
contains a point at which $\widehat g$ attains the global minimum $m(\varepsilon)$.

Let $(a,a)\in\overline\gamma\cap \{(z,z)|\> z\in A \}$.
Applying Proposition~\ref{prop:one-dim} to $\gamma$, after translation by $(a,a)$,
we obtain $N_0\in\mathbb N$ such that, for every point $(x,y)\in\gamma$ sufficiently close to $(a,a)$,
$$
\widehat g(x,y) \ge \frac{1}{ e_{r+s}\left( \|(x,y)-(a,a) \|^{-N_0}\right)}.
$$
Since
$\|x-y \|\le\sqrt{2}\|(x,y)-(a,a) \|$, after increasing $N_0$, we obtain some $N\in\mathbb N$ such that
$$
m(\varepsilon) \ge
\frac{1}{e_{r+s}(\varepsilon^{-N})}
$$
for all sufficiently small $\varepsilon>0$.
It remains to pass from $W_\varepsilon$ to the original set $X$.
Let a point $x\in X\cap \{g>0 \}$ be sufficiently close to $A$, and put $\varepsilon:=\operatorname{dist}(x,A)$.
Choose a nearest point $a\in A$, so that $\| x-a \|=\varepsilon$.
By Lemma~\ref{le:lemma}, the vector $x-a\in C_aA$.
Therefore $(x,a)\in W_\varepsilon$.
Hence, by the definition of $m(\varepsilon)$, we have $g(x)=\widehat g(x,a)\ge m(\varepsilon)$.
It follows that,
$$
g(x)\ge \frac{1}{e_{r+s}\left( \operatorname{dist}(x,A)^{-N} \right)}
$$
for some $N\in\mathbb N$ and the neighbourhood $U= \{ x \in X \cap \{ g>0 \} |\ {\rm dist} (x,A) < \eps_0 \}$
for a sufficiently small $\eps_0 >0$.
This proves the corollary.
\end{proof}

\end{document}